\documentclass[12pt]{amsart}
\usepackage{amscd}
\usepackage{amsmath}
\usepackage{amsxtra}
\usepackage{amsfonts}
\usepackage{amssymb}
\usepackage{amsthm}
\usepackage{inputenc}
\usepackage[english]{babel}
\usepackage{blindtext}
\usepackage{bbm}
\usepackage{tikz}
\usepackage{float}
\usetikzlibrary{arrows}
\newtheorem{theorem}{Theorem}[section]
\newtheorem{corollary}[theorem]{Corollary}
\newtheorem{lemma}[theorem]{Lemma}
\newtheorem{proposition}[theorem]{Proposition}

\theoremstyle{definition}
\newtheorem{definition}[theorem]{Definition}
\newtheorem{remark}[theorem]{Remark}

\newtheorem{example}[theorem]{Example}
\theoremstyle{remark}

\renewcommand{\theclaim}{\textup{\theclaim}}

\numberwithin{equation}{section}

\def\openone

{\mathchoice

{\hbox{\upshape \small1\kern-3.3pt\normalsize1}}

{\hbox{\upshape \small1\kern-3.3pt\normalsize1}}

{\hbox{\upshape \tiny1\kern-2.3pt\SMALL1}}

{\hbox{\upshape \Tiny1\kern-2pt\tiny1}}}

\makeatletter

\newbox\ipbox

\newcommand{\ip}[2]{\left\langle #1\, , \,#2\right\rangle}
\newcommand{\diracb}[1]{\left\langle #1\mathrel{\mathchoice

{\setbox\ipbox=\hbox{$\displaystyle \left\langle\mathstrut
#1\right.$}

\vrule height\ht\ipbox width0.25pt depth\dp\ipbox}

{\setbox\ipbox=\hbox{$\textstyle \left\langle\mathstrut
#1\right.$}

\vrule height\ht\ipbox width0.25pt depth\dp\ipbox}

{\setbox\ipbox=\hbox{$\scriptstyle \left\langle\mathstrut
#1\right.$}

\vrule height\ht\ipbox width0.25pt depth\dp\ipbox}

{\setbox\ipbox=\hbox{$\scriptscriptstyle \left\langle\mathstrut
#1\right.$}

\vrule height\ht\ipbox width0.25pt depth\dp\ipbox}

}\right. }

\newcommand{\dirack}[1]{\left. \mathrel{\mathchoice

{\setbox\ipbox=\hbox{$\displaystyle \left.\mathstrut
#1\right\rangle$}

\vrule height\ht\ipbox width0.25pt depth\dp\ipbox}

{\setbox\ipbox=\hbox{$\textstyle \left.\mathstrut
#1\right\rangle$}

\vrule height\ht\ipbox width0.25pt depth\dp\ipbox}

{\setbox\ipbox=\hbox{$\scriptstyle \left.\mathstrut
#1\right\rangle$}

\vrule height\ht\ipbox width0.25pt depth\dp\ipbox}

{\setbox\ipbox=\hbox{$\scriptscriptstyle \left.\mathstrut
#1\right\rangle$}

\vrule height\ht\ipbox width0.25pt depth\dp\ipbox}

} #1\right\rangle}

\newcommand{\cj}[1]{\overline{#1}}

\newcommand{\bz}{\mathbb{Z}}

\newcommand{\br}{\mathbb{R}}
\newcommand{\bc}{\mathbb{C}}

\newcommand{\bn}{\mathbb{N}}

\newcommand{\beq}{\begin{equation}}
\newcommand{\eeq}{\end{equation}}

\def\blfootnote{\xdef\@thefnmark{}\@footnotetext}

\renewcommand{\mod}{\operatorname{mod}}

\input xy
\xyoption{all}
\usepackage{amssymb}

\def\-{^{-1}}

\def\g{\mathcal{G}}

\def\ty{\emptyset}

\begin{document}
\date{}

\title{Weighted Fourier Frames on Self-Affine Measures}
\author{Dorin Ervin Dutkay}

\address{[Dorin Ervin Dutkay] University of Central Florida\\
	Department of Mathematics\\
	4000 Central Florida Blvd.\\
	P.O. Box 161364\\
	Orlando, FL 32816-1364\\
U.S.A.\\} \email{Dorin.Dutkay@ucf.edu}

\author{Rajitha Ranasinghe}

\address{[Rajitha Ranasinghe] University of Central Florida\\
	Department of Mathematics\\
	4000 Central Florida Blvd.\\
	P.O. Box 161364\\
	Orlando, FL 32816-1364\\
U.S.A.\\} \email{rajitha13@knights.ucf.edu }

\thanks{} 

\begin{abstract}
Continuing the ideas from our previous paper \cite{DR16}, we construct Parseval frames of weighted exponential functions for self-affine measures.
 
\end{abstract}
\maketitle \tableofcontents

\section{Introduction}

A probability measure $\mu$ on $\br$ is called \emph{spectral} if there exists a sequence of exponential functions which form an orthonormal basis for $L^2(\mu)$. Of course, the main example is the Lebesgue measure on the unit interval with the classical Fourier series. In 1998, Jorgensen and Pedersen \cite{JP98} constructed the first example of a singular, non-atomic spectral measure, based on a Cantor set with scale 4. Since then, many other examples of spectral singular measures have been constructed (see e.g., \cite{Str00, LW02, DJ06,DJ07}), most of them are based on affine iterated function systems (see Definition  \ref{defns}). In the same paper, Jorgensen and Pedersen showed that the Hausdorff measure on the Middle Third Cantor set is not spectral and Strichartz \cite{Str00} posed the question whether there are any frames of exponential functions for the Middle Third Cantor set. As far as we know, this question is still open.

In search of a frame for the Middle Third Cantor set, in \cite{PW15}, Picioroaga and Weber introduced an interesting idea for the construction of weighted exponential frames (also called weighted Fourier frames)  for the self-affine measures, in particular for the Cantor set $C_4$ in Jorgensen and Pedersen's example. The word ``weighted'' means that the exponential function is multiplied by a constant. The basic idea is to use Cuntz algebras to construct an orthonormal set for a dilation of the Hilbert space of the fractal measure, which then projects into a Parseval frame of weighted exponential functions. In \cite{DR16}, the authors generalized the aforementioned idea of Picioroaga and Weber to construct Parseval Fourier frames for self-affine measures (see Definition \ref{defns}). 

The present paper is a continuation of the paper \cite{DR16}. Here we refine the main result in \cite{DR16} by removing some conditions in the hypothesis and constructing new families of weighted Fourier frames for self-affine measures. 

We begin with some definitions and we recall the main notions that allow us to formulate the main result. In section 2 we present the proof of the result and in section 3 we present several examples.

\begin{definition}\label{defns}

For a given integer $R\geq 2$ and a finite set of integers $B$ with cardinality $|B|=:N,$ we define the \emph{affine iterated function system} (IFS) $\tau_{b}(x)=R^{-1}(x+b), x \in \br, b \in B.$ The \emph{self-affine measure} (with equal weights) is the unique probability measure $\mu = \mu(R, B)$ satisfying 
\begin{equation}\label{self-affine}
\mu(E)=\frac{1}{N} \sum_{b \in B} \mu(\tau_{b}^{-1}(E)),~~\textnormal{for all Borel subsets}~E~\textnormal{of}~\mathbb{R}.
\end{equation}
This measure is supported on the {\it attractor} $X_B$ which is the unique compact set that satisfies
$$
X_B= \bigcup_{b\in B} \tau_b(X_B).
$$
The set $X_B$ is also called the {\it self-affine set} associated with the IFS, and it can be described as 
$$X_B=\left\{\sum_{k=1}^\infty R^{-k}b_k : b_k\in B\right\}.$$

One can refer to \cite{Hut81} for a detailed exposition of the theory of iterated function systems. We say that $\mu = \mu(R,B)$ satisfies the {\it no overlap condition} if
$$
\mu(\tau_{b}(X_B)\cap \tau_{b'}(X_B))=0, \ \forall~b\neq b'\in B.
$$
For $\lambda\in \mathbb{R}$, define
$$e_\lambda(x)=e^{2\pi i\lambda x},~~ (x\in \mathbb{R}).$$
For a Borel probability measure $\mu$ on $\br$ we define its \emph{Fourier transform} by
$$\widehat\mu(t)=\int e^{2\pi it x}\,d\mu,\quad(t\in \mathbb{R}).$$

A \emph{frame} for a Hilbert space $H$ is a family $\{e_i\}_{i\in I}\subset H$ such that there exist constants $A,B>0$ such that for all $v\in H$,
$$A\|v\|^2\leq \sum_{i\in I}|\ip{v}{e_i}|^2\leq B\|v\|^2.$$
The largest $A$ and smallest $B$ which satisfy these inequalities are called the \emph{frame bounds}. The frame is called a \emph{Parseval} frame if both frame bounds are $1$. 
\end{definition}

\noindent
{\bf Assumptions 1.1.}

Assume that there exists a finite set $L \subset \mathbb{Z}$ with $0 \in L, |L|=:M$ and complex numbers $\left( \alpha_l \right)_{l \in L}$ such that the following properties are satisfied:
\begin{enumerate}
\item $\alpha_{0}=1.$
\item The matrix
\begin{equation}\label{matrix_T}
T := \frac{1}{\sqrt{N}} \left( e^{2 \pi i R^{-1}l \cdot b}\alpha_l \right)_{l \in L, b \in B}
\end{equation}
is an isometry, i.e., $T T^{*}=I_{N},$ i.e., its columns are orthonormal, which means that 
\begin{equation}
\frac{1}{N}\sum_{l\in L}|\alpha_l|^2e^{2\pi i R^{-1}l\cdot (b-b')}=\delta_{b,b'},\quad (b,b'\in B).
\label{eq2.2}
\end{equation}
\end{enumerate}

\bigskip
\begin{definition}\label{defcong}
For $k\in\bz$, we denote 
$$[k]:=\{k'\in\bz : (k'-k)\cdot R^{-1}b\in\bz, \mbox{ for all } b\in B\}.$$
We denote by $[L]:=\{[l] : l\in L\}$. 
\end{definition}

\bigskip
\begin{definition}\label{def1}

Let 
\begin{equation}
m_B(x)=\frac{1}{N}\sum_{b\in B}e^{2\pi i bx},\quad(x\in\br).
\label{eqmb}
\end{equation}
With the notations of Theorem 1.4 in \cite[p.1606]{DR16} a set $\mathcal M\subset \br$ is called {\it invariant} if for any point $t\in \mathcal M$, and any $l\in L$, if $\alpha_{l}m_B((R^T)^{-1}(t-l))\neq 0$, then $g_{l}(t):=(R^T)^{-1}(t-l)\in \mathcal M$. $\mathcal M$ is said to be non-trivial if $\mathcal M\neq \{0\}$.  We call a finite {\it minimal invariant} set a {\it min-set}.

Note that 
\begin{equation}
\sum\limits_{l\in L} | \alpha_{l}  |^2 \ |m_{B}(g_{l} (t))|^{2}=1 \quad (t \in \mathbb{R}^d ),
\end{equation}
(see (3.2) in \cite[p.1615]{DR16}), and therefore, we can interpret the number $| \alpha_{l}  |^2 \ |m_{B}(g_{l} (t))|^{2}$ as the probability of transition from $t$ to $g_{l}(t)$, and if this number is not zero then we say that this {\it transition is possible in one step (with digit $l$)}, and we write $t\rightarrow g_{l}(t)$ or $t\stackrel{l}{\rightarrow}g_l(t)$. We say that the {\it transition is possible} from a point $t$ to a point $t'$ if there exist $t_0=t$, $t_1,\dots, t_n=t'$ such that $t=t_0\rightarrow t_1\rightarrow\dots\rightarrow t_n=t'$. The {\it trajectory} of a point $t$ is the set of all points $t'$ (including the point $t$) such that the transition is possible from $t$ to $t'$.

A {\it cycle} is a finite set $\{t_0,\dots,t_{p-1}\}$ such that there exist $l_0,\dots, l_{p-1}$ in $L$ such that $g_{l_0}(t_0)=t_1,\dots, g_{l_{p-1}}(t_{p-1})=t_{p}:=t_0$. Points in a cycle are called {\it cycle points}. 

A cycle $\{t_0,\dots, t_{p-1}\}$ is called {\it extreme} if $|m_B(t_i)|=1$ for all $i$; by the triangle inequality, since $0\in B$, this is equivalent to $t_i\cdot b\in \bz$ for all $b\in B$. 
\end{definition}

The next proposition gives some information about the structure of finite, minimal sets, which makes it easier to find such sets in concrete examples.

\begin{proposition}\label{pr3}\cite{DR16}
Assume $\alpha_l\neq 0$ for all $l\in L$. 
Let $\mathcal M $ be a non-trivial finite, minimal invariant set. Then, for every two points $t,t'\in \mathcal M $ the transition is possible from $t$ to $t'$ in several steps. In particular, every point in the set $\mathcal M $ is a cycle point. The set $\mathcal M $ is contained in the interval $\left[\frac{\min(-L)}{R-1},\frac{\max(-L)}{R-1}\right]$. 

If $t$ is in $\mathcal M $ and if there are two possible transitions $t\rightarrow g_{l_1}(t)$ and $t\rightarrow g_{l_2}(t)$, then $l_1\equiv l_2(\mod R)$.

Every point $t$ in $\mathcal M $ is an extreme cycle point, i.e., $|m_B(t)|=1$ and if $t\rightarrow g_{l_0}(t)$ is a possible transition in one step, then $[l_0]\cap L=\{l\in L : l\equiv l_0(\mod R)\}$ (with the notation in Definition \ref{defcong}) and
\begin{equation}
\sum_{l\in L,l\equiv l_0(\mod R)}|\alpha_l|^2=1.
\label{eqac}
\end{equation}

In particular $t\cdot b\in\bz$ for all $b\in B$. 
\end{proposition}

\medskip
\noindent
\begin{definition}
Let $c$ be an extreme cycle point in some finite minimal invariant set. A word $l_0 \dots l_{p-1}$ in $L$ is called a \emph{cycle word} for $c$ if $g_{l_{p-1}} \dots g_{l_{0}}(c)=c$
and $g_{l_k} \dots g_{l_{0}}(c) \neq c$ for $0 \leq k < p-1,$ and the transitions $c \rightarrow g_{l_{0}}(c) \rightarrow g_{l_{1}}g_{l_{0}}(c) \rightarrow \dots \rightarrow g_{l_{p-2}}\dots g_{l_{0}}(c)\rightarrow  g_{l_{p-1}}\dots g_{l_{0}}(c)=c$ are possible. 
\end{definition}

For every finite minimal invariant set $\mathcal{M},$ pick a point $c(\mathcal{M})$ in $\mathcal{M}$ and define $\Omega(c(\mathcal{M}))$ to be the set of finite words with digits in $L$ that do not end in a cycle word for $c(\mathcal{M})$, i.e., they are not of the form $\omega\omega_0$ where $\omega_0$ is a cycle word for $c$ and $\omega$ is an arbitrary word with digits in $L$.

\medskip
\noindent
\begin{theorem}\label{th1.6}
Suppose $(R,B,L)$ and $(\alpha_l)_{l\in L}$ satisfy the Assumptions 1.1. Then the set 
$$\Bigg\{ \left( \prod_{j=0}^{n} \alpha_{l_j} \right) e_{l_{0}+R l_{1} + \dots + R^k l_k + R^{k+1} c(\mathcal{M}) }  : l_0 \dots l_n \in \Omega(c(\mathcal{M})), \mathcal{M}~\textnormal{is a min-set}   \Bigg\}$$
is a Parseval frame for $L^2(\mu(R,B)).$
\end{theorem}

Theorem \ref{th1.6} improves Theorem 1.4 in \cite{DR16}, where we assumed that there are no non-trivial min-sets. As we see here, if there are some non-trivial min-set, then each such set has some contribution to the Parseval frame of exponential function.

\bigskip
\section{Proofs}
The beginning of the proof follows the steps in the proof of Theorem 1.4 in \cite{DR16}. 
Consider the system with scaling $R' = N '$ where $N' \in \mathbb{Z}$ and $N N' \geq M,$ and digits $B'=\{ 0, 1, \dots , N'-1  \}.$ Define the iterated function system
$$\tau_{b'}(x')={R'}^{-1}(x' + b') ~~~~~~~ (x' \in \mathbb{R}, b' \in B').$$
The invariant measure for this iterated function system is the Lebesgue measure on $[0,1]$ and we denote it by $\mu' .$ The attractor of this iterated function system is $X_{B'}=[0,1].$ We can identify $L$ with a subset $L'$ of $B \times B'$ by some injective function $\iota ,$ in such a way that $0$ from $L$ corresponds to $(0,0)$ from $B \times B',$ and we define $l(b,b')=l$ if $(b,b')=\iota(l),$ $l(b,b')=0$ if $(b,b') \notin L',$ and $\alpha_{(b,b')}=\alpha_{l}$ if $(b,b')=\iota(l)$ and $\alpha_{(b,b')}=0$ if $(b,b') \notin L'.$ 

In other words, we complete the matrix $T$ in \eqref{matrix_T} with some zero rows, so that the rows are now indexed by $B \times B'.$ Of course, the isometry property is preserved, and $\alpha_{(0,0)}=0,~l(0,0)=0.$

As in pages 1607-1609 in \cite{DR16}, one can construct numbers $a_{(b,b'), (c,c')},~(b,b'), (c,c') \in B \times B'$  with the following properties:
\begin{enumerate}
\item The matrix 
\begin{equation}
\frac{1}{\sqrt{N N'}} \left(a_{(b,b'), (c,c')} e^{2 \pi i R^{-1} l(b,b') \cdot c} \right)_{(b,b'), (c,c') \in B \times B'}, 
\end{equation}
is unitary and the first row, corresponding to $(b,b')=(0,0)$, is constant $\frac{1}{\sqrt{N N'}}.$ So $a_{(0,0), (c,c')}=1$  for all $(c,c') \in B \times B' .$

\item For all $(b,b') \in B \times B', c \in B,$
\begin{equation}
\frac{1}{N'} \sum_{c \in B'} a_{(b,b'), (c,c')}=\alpha_{(b,b')}.
\end{equation}

\end{enumerate}

Next, we construct some Cuntz isometries $S_{(b,b')}, ~(b,b') \in B \times B'$ in the dilation space $L^2 (\mu \times \mu'),$ and with them, we construct an orthonomal set, by applying the Cuntz isometries to the functions $e_{c}$ for points $c$ in each min-set.

Recall that some operators $\{ S_{(b,b')} : (b,b')\in B\times B'\}$ on a Hilbert space $H$, are called Cuntz isometries if they satisfy the relations
$$S_{(b,b')}^*S_{(c,c')}=\delta_{(b,b'),(c,c')}I_H,\quad \sum_{(b,b')\in B\times B'}S_{(b,b')}S_{(b,b')}^*=I_H.$$

Next, we claim that the measure $\mu_B$ has no overlap. This follows from \cite[Theorem 2.2 and Proposition 2.3]{DHL16}, if we show that the elements in $B$ are not congruent modulo $R$. But if $b\equiv b'(\mod R)$, for $b\neq b'$ in $B$, then, using \eqref{eq2.2}, we get that 
$$\sum_{l\in L}|\alpha_l|^2=0,$$
a contradiction. 

Define now, the maps $\mathcal R: X_{B} \rightarrow X_{B}$ by 
\begin{equation}
\mathcal R x = Rx - b,~~\textnormal{if}~x \in \tau_{b}(X_B),
\end{equation}
and $\mathcal R': X_{B'} \rightarrow X_{B'}$ by
\begin{equation}
\mathcal R'x' = R'x' - b',~~\textnormal{if}~x' \in \tau_{b'}(X_{B'}).
\end{equation}

Note that $\mathcal R(\tau_b x) =x$ for all $x \in X_B$ and $\mathcal R'(\tau_b' x') =x'$ for all $x' \in X_B' .$ The no-overlap condition guarantees  that the maps are well defined. 

Next, we consider the cartesian product of the two iterated function systems and define the maps
\begin{equation}
\Upsilon_{(b,b')}(x,x')=\left( R^{-1}(x+b) , {R'}^{-1}(x'+b') \right), 
\end{equation}
for $(x,x') \in R \times R'$ and $(b,b') \in B \times B'.$ Note that the measure $\mu \times \mu'$ is the invariant measure of the iterated function system $\left( \Upsilon_{(b,b')} \right)_{(b,b') \in B \times B'}.$ Define the functions
\begin{equation}
m_{(b,b')}(x,x') = e^{2 \pi i l(b,b') \cdot x} H_{(b,b')}(x,x'),
\end{equation}
for $(x,x') \in R \times R', (b,b') \in B \times B',$ where 
$$H_{(b,b')}(x,x')=\sum_{(c,c') \in B \times B'} a_{(b,b'), (c,c')} \chi_{\Upsilon_{(c,c')}(X_B \times X_{B'})} (x,x') .$$
($\chi_{A}$ denotes the characteristic function of the set $A.$)
With these filters we define the operators $S_{(b,b')}$ on $L^2 (\mu \times \mu')$ by 
\begin{equation}
\left( S_{(b,b')} f \right)(x,x') = m_{(b,b')}(x,x') f(\mathcal Rx, \mathcal R'x') .
\end{equation}

\bigskip
\begin{lemma} \textnormal{(\cite{DR16}, Lemma 2.2.)} 
The operators $S_{(b,b')},~(b,b') \in B \times B'$ are a representation of the Cuntz algebra $\mathcal{O}_{NN'}.$ The adjoint $S^{*}_{(b,b')}$ is given by the formula
\begin{equation}
\left( S_{(b,b')}^{*} f \right)(x,x')=\frac{1}{NN'}\sum_{(c,c') \in B \times B'}\cj m_{(b,b')} \left( \Upsilon_{(c,c')}(x,x') \right) f(\Upsilon_{(c,c')} (x,x')),
\end{equation}
for $f \in L^2 (\mu \times \mu'),~ (x,x') \in X_{B} \times X_{B'}.$
\end{lemma}

For a word $\omega = (b_1 , b_{1}^{\prime}) (b_2 , b_{2}^{\prime}) \cdots (b_k , b_{k}^{\prime})$ and a point $c$ in some min-set, we compute $\left( S_{\omega}e_{c} \right)(x,x')$ (here $e_{c}(x,x') = e^{2 \pi i c \cdot x} $ and $S_\omega=S_{(b_1,b_1')}\dots S_{(b_k,b_k')}$). 
\begin{equation*}
\begin{split}
&\left( S_{\omega} e_{c} \right) (x,x')=S_{(b_1,b_{1}')} \cdots S_{(b_{k-1} , b_{k-1}')} e^{2 \pi i l (b_k , b_{k}') \cdot x} H_{(b_{k} , b_{k}')}(x,x') e^{2 \pi i c \cdot\mathcal R x} \\ 
&=S_{(b_1,b_{1}')} \cdots S_{(b_{k-2} , b_{k-2}')} e^{2 \pi i l (b_{k-1} , b_{k-1}') \cdot x} e^{2 \pi i l (b_k , b_{k}') \cdot\mathcal R x} e^{2 \pi i c \cdot\mathcal R^2 x} \cdot  H_{(b_{k-1} , b_{k-1}')}(x,x')   H_{(b_{k} , b_{k}')}(\mathcal Rx,\mathcal R'x') \\ 
&=\dots~~~~ \dots~~~~~ \dots \\ 
&=e^{2 \pi i \left( l (b_1 , b_{1}') \cdot x + l (b_2 , b_{2}') \cdot \mathcal Rx + \cdots + l (b_{k} , b_{k}') \cdot \mathcal R^{k-1} x + c \cdot \mathcal R^k x \right) } \cdot H_{(b_{1} , b_{1}')}(x,x')   H_{(b_{2} , b_{2}')}(\mathcal Rx,\mathcal R'x') \cdots \\
& ~~~~~~~~~~~~~~~~~~~~~~~~~~~~~~~~~~~~~~~~~ \cdots  H_{(b_{k-1} , b_{k-1}')}(\mathcal R^{k-1}x,\mathcal R'^{k-1}x').
\end{split}
\end{equation*}

Since $l(b,b') \cdot R^k b^{''} \in \mathbb{Z}$ for all $(b,b') \in B \times B',~b'' \in B, k \geq 0$ and since, by Proposition \ref{pr3}, $c \cdot R^k b$ for all $b \in B, k \geq 0,$ we get that the first term in the product above is 
$$e^{2 \pi i\left( l (b_1 , b_{1}') + Rl (b_2 , b_{2}')  + \cdots +R^{k-1} l (b_{k} , b_{k}') +R^k c  \right) \cdot x} .$$

Next we compute the projection $P_V S_{\omega} e_{c}$ onto the subspace
$$V=\{ f(x,y) = g(x) : g \in L^2 (\mu) \}.$$

As in \cite[p.8]{DR16}, we obtain
\begin{equation}\label{eq2.9}
P_V S_{\omega} e_{c} = e^{2 \pi i \left( l (b_1 , b_{1}') + Rl (b_2 , b_{2}')  + \cdots +R^{k-1} l (b_{k} , b_{k}') +R^k c  \right) \cdot x} \prod_{j=1}^{k} \alpha_{(b_j , b_j')}. 
\end{equation}
Also, as in equations $(2.12)$ and  $(2.14)$ in \cite{DR16},  we have for $(b,b') \in B \times B', t \in \mathbb{R},$
\begin{equation}
S_{(b,b')}^{*} e_{t} = \overline{\alpha}_{(b,b')} m_{B}\left( g_{(b,b')} (t)  \right) e_{g_{(b,b')} (t)},
\end{equation}
$$g_{(b,b')}(t)=R^{-1} \left( x - l(b,b') \right),$$
and
\begin{equation}\label{eq2.12}
\sum_{(b,b') \in B \times B'} |\alpha_{(b,b')}|^2 |m_B\left( g_{(b,b')}(t) \right)|^2 = 1~~~(t \in \mathbb{R}).
\end{equation}

Note that $\alpha_{(b,b')}=0$ if $(b,b')$ is not in $\iota(L).$ So, for such $(b,b')$ the transitions $x \rightarrow g_{(b,b')}(x)$ are not possible. Therefore we can work with the set $B \times B'$ when we talk about transitions and min-sets.

For a min-set $\mathcal{M},$ we can identify $\Omega(c(\mathcal{M}))$ with the set of words $\beta_1 \dots \beta_k$ with digits in $B \times B'$ that do not end in a cycle word for $c(\mathcal{M}).$

We claim that
\begin{equation}\label{eqn_1}
\mathcal{E}=\{ S_\omega e_{c(\mathcal{M})} : \omega \in \Omega(c(\mathcal M)),~~\mathcal{M}~\textnormal{is a min-set} \}
\end{equation}
is an orthonormal family in $L^2 (\mu \times \mu').$

Take two min-sets $\mathcal{M}, \mathcal{M}'$ and two words $\omega \in \Omega(c(\mathcal{M}))$, $\omega' \in \Omega(c(\mathcal{M}'))$, $\omega \neq \omega'$.  If $\omega=\omega_1 \dots \omega_n , \omega'=\omega_{1}' \dots \omega_{m}'$ and there exists $1 \leq i \leq n, m$ such that $\omega_i \neq \omega_{i}'$, then take the first such $i$. Since $S_{\omega_i}, S_{\omega_{i}'}$ have orthogonal  ranges and the $S_{\omega_j}$ are isometries, it follows that $S_{\omega}e_{c(\mathcal{M})} \perp S_{\omega'}e_{c(\mathcal{M}')}.$ The remaining case is when $\omega$ is a prefix of $\omega'$ or vice versa. Assume $\omega$ is a prefix of $\omega',$
$$\omega'=\omega  \beta , ~~~\beta=\beta_1 \dots \beta_n.$$
Then 
$$
\ip{S_{\omega}e_{c(\mathcal{M})}}{S_{\omega \beta} e_{c(\mathcal{M}')}} =\ip{ e_{c(\mathcal{M})}}{  S_{\beta} e_{c(\mathcal{M}')}}
=\ip{S_{\beta_{n}^{*}} \dots S_{\beta_{1}^{*}}e_{c(\mathcal{M})}}{e_{c(\mathcal{M}')}} $$$$
=\overline{\alpha}_{\beta_1} m_{B}(g_{\beta_1}(c(\mathcal{M}))) \ip{S_{\beta_{n}^{*}} \dots S_{\beta_{2}^{*}}e_{g_{\beta_1}(c(\mathcal{M}))}}{e_{c(\mathcal{M}')}}
=\dots $$$$
=\overline{\alpha}_{\beta_1} m_{B}(g_{\beta_1}(c(\mathcal{M}))) \overline{\alpha}_{\beta_2} m_{B}(g_{\beta_2}g_{\beta_1}(c(\mathcal{M}))) \dots \overline{\alpha}_{\beta_n} m_{B}(g_{\beta_n} \dots g_{\beta_1}(c(\mathcal{M}))) \cdot \ip{ e_{g_{\beta_n} \dots g_{\beta_1}}(c(\mathcal{M})) }{ e_{c(\mathcal{M}')}}.
$$
Thus, if the original inner product is non-zero, then the transitions 
$$c(\mathcal{M}) \rightarrow g_{\beta_1}(c(\mathcal{M})) \rightarrow g_{\beta_2}g_{\beta_1}(c(\mathcal{M})) \rightarrow g_{\beta_n} \dots g_{\beta_1}(c(\mathcal{M}))$$
are possible and 
$$\langle e_{g_{\beta_n} \dots g_{\beta_1}(c(\mathcal{M})))} , e_{c(\mathcal{M})')} \rangle \neq 0.$$

But if the transitions are possible, then 
$$g_{\beta_n} \dots g_{\beta_1}(c(\mathcal{M})) =: c''$$
is another point in $\mathcal{M}.$
We will show that if $c \neq c'$ are two cycle points (could be from the same min-set), then $\langle e_c , e_{c'} \rangle = 0.$
If we have this, then if 
$$\langle e_{g_{\beta_n} \dots g_{\beta_1}(c(\mathcal{M}))} , e_{c(\mathcal{\mathcal{M}}')} \rangle \neq 0,$$
then 
$$\mathcal M \ni g_{\beta_n} \dots g_{\beta_1}(c(\mathcal{M})) =c(\mathcal{M}') \in \mathcal{M}'.$$
So $\mathcal{M} = \mathcal{M}'$, and $c(\mathcal M)=c(\mathcal M')$, and this means that $\beta_1 \dots \beta_n$ is, or ends in, a  cycle word for $c(\mathcal{M})$ which contradicts the fact that $\omega' = \omega \beta \in \Omega(c(\mathcal{M}'))=\Omega(c(\mathcal{M})).$ So it remains to show that $\langle e_{c} , e_{c'} \rangle = 0$ for two distinct extreme cycle points $c \neq c'.$

Assume $\ip{ e_{c}}{  e_{c'}}  \neq 0.$ Then (see e.g., \cite[Equation (2.5)]{DJ07b})
$$0 \neq \ip{ e_{c}}{  e_{c'}}= \hat{\mu}(c-c') = \prod_{k=1}^{\infty} m_B\left( \frac{c-c'}{R^k} \right).$$
So $m_B\left( \frac{c-c'}{R^k} \right) \neq 0$ for all $k \geq 1.$

By repeating a cycle word for $c'$ as many times as needed we can find a word $\beta_1 \dots \beta_n$ with $n$ as large as we want, such that $g_{\beta_n} \dots g_{\beta_1} (c') = c'.$ Then $c'=l(\beta_1)+R l(\beta_2) + \dots +R^{n-1} l(\beta_n)+R^n c'.$ So
$$\frac{c-c'}{R^k}=g_{\beta_k} \dots g_{\beta_1}(c) + l(\beta_{k+1}) + R l(\beta_{k+2}) + \dots + R^{n-k-1} l(\beta_{n})+R^{n-k} c'.$$
Since $l(\beta_i) \in \mathbb{Z}$ and $b \cdot c' \in \mathbb{Z}$ for all $b \in B ,$ it follows that 
$$ 0 \neq m_B\left( \frac{c-c'}{R^k} \right) = m_B(g_{\beta_k} \dots g_{\beta_1}(c)),~~~\textnormal{for all}~~~ k \leq n.$$
So we can conclude that the transitions 
$$c \rightarrow g_{\beta_1}(c) \rightarrow g_{\beta_2}g_{\beta_1}(c) \rightarrow \dots \rightarrow g_{\beta_n} \dots g_{\beta_1}(c)$$
are possible. (The numbers $\alpha_{\beta_k}$ are non-zero because $\beta_1 \dots \beta_n$ is a cycle word for $c',$ so the transitions with this same digits are possible for $c'.$) 

If we take $\beta_1 \dots \beta_{k_n} = \underbrace{\omega' \omega' \dots \omega'}_\text{$n$ terms},$ where $\omega'$ is a cycle word for $c',$ then $\lim_{n \rightarrow \infty} g_{\beta_{k_n}} \dots g_{\beta_1} (c)=c'.$
But the points $g_{\beta_{k_n}} \dots g_{\beta_1}(c)$ all lie in the same min-set, which is finite, so we must have $g_{\beta_{k_n}} \dots g_{\beta_1} (c)=c',$ for some $n.$ Since $\beta_1 \dots \beta_{k_n}=\omega' \omega' \dots \omega',$ we get that $c=c',$ a contradiction. Thus $\ip{e_c}{e_c'}=0$, and we can conclude that the family \eqref{eqn_1} is orthonormal. Then, if we show that this orthonormal family projects onto a complete family in $V$, using a well known result in frame theory (see \cite{AA95} or \cite[Lemma 1.2]{DR16}), we obtain that its projection onto $V$ is Parseval frame. 

Thus, in what follows, we show that the projection of this orthonormal family onto the subspace $V$ is complete in $V$. We begin again by using the ideas from \cite{DR16}, but significant extra work is needed. 

Define the function $h:\mathbb{R} \rightarrow \mathbb{R}$ by 
\begin{equation}\label{eqn_2}
h(t)=\sum_{v \in \mathcal{E}} |\langle e_t , v  \rangle|^2 = \| P_{\mathcal{K}} e_t \|^2 , 
\end{equation}
where $P_{\mathcal{K}}$ is the projection onto the span $\mathcal{K}$ of $\mathcal{E}.$ Note that $h(c(\mathcal{M}))=1$ for all min-sets $\mathcal{M},$ since $e_{c(\mathcal{M})} \in \mathcal{E}.$ 

Note also that $0 \leq h\leq 1$ and we claim that $h$ can be extended to an entire function on $\mathbb{C}$.

 Let $c :=c(\mathcal{M}).$ For a fixed $\omega \in \Omega(c(\mathcal M)),$ define $f_{\omega} : \mathbb{C} \rightarrow \mathbb{C}$ by 
\begin{equation*}
\begin{split}
f_{\omega}(z) &=\int e^{2 \pi i z x} \overline{ \left( S_{\omega} e_{c} \right) } (x,x')~d\mu(x) d\mu'(x') \\
&=\langle e_{t} , S_{\omega} e_{c}  \rangle.
\end{split}
\end{equation*}
Since the measure is compactly supported, a standard convergence argument shows that the function $f_{\omega}$ is entire. Similarly $f_{\omega}^{*}(t) :=\overline{f_{\omega}(\overline{z})}$ is entire and for real $t,$
\begin{equation*}
\begin{split}
f_{\omega}(t) f_{\omega}^{*}(t)&=\langle e_{t} , S_{\omega} e_{c}  \rangle \cdot \langle \overline{e_{t} , S_{\omega} e_{c}} \rangle \\
&=|\langle e_{t} , S_{\omega} e_{c}  \rangle|^2 .
\end{split}
\end{equation*}
Thus 
$$h(t)=\sum_{\mathcal{M}} \sum_{\omega \in \Omega(c(\mathcal M))} f_{\omega}(t) f_{\omega}^{*}(t),\quad(t\in\br).$$
For $n \in \mathbb{N},$ let 
$$h_n(t)=\sum_{\mathcal{M}} \sum_{\omega\in\Omega(c(\mathcal M)),|\omega| \leq n} f_{\omega}(t) f_{\omega}^{*}(t),\quad(t\in\bc)$$ 
which is entire. By H\"{o}lder's inequality, for $t\in\bc$,
\begin{equation*}
\begin{split}
\sum_{\mathcal{M}, \omega} |f_{\omega}(t) f_{\omega}^{*}(t)| &\leq \left(\sum_{\mathcal{M}, \omega} |\langle e_{t} , S_{\omega} e_{c}  \rangle|^2 \right)^{1/2} \cdot \left(\sum_{\mathcal{M}, \omega} |\langle e_{\overline{t}} , S_{\omega} e_{c}  \rangle|^2 \right)^{1/2} \\ 
&\leq \| e_t \| \| e_{\overline{t}} \| \leq e^{K |t|},
\end{split}
\end{equation*}
for some constant $K.$ Thus the sequence $h_n (t)$ converges pointwise to a function $h(t)$ and is uniformly bounded on bounded sets. By the theorems of Montel and Vitali, the limit function is entire and it extends $h$ from $\br$ to $\bc$. 

Next we show that
\begin{equation}\label{eqn_3}
\sum_{\omega \in \Omega(c(\mathcal M))} |\langle S_{\omega} e_{c(\mathcal{M})}~,~e_{t}  \rangle|^2 = \sum_{\beta_1 \in B \times B'} \sum_{\omega \in \Omega(c(\mathcal M))} |\langle S_{\beta_1} S_{\omega} e_{c(\mathcal{M})}~,~e_{t}  \rangle|^2 .
\end{equation}
If $\omega \in \Omega(c(\mathcal M))$ and $\omega \neq \ty$ (the empty word), then $\omega = \beta_1 \omega'$ and $\omega' \in \Omega(c(\mathcal M)).$  Also, if $\beta_1 \in B \times B'$ and $\omega' \in \Omega(c(\mathcal M))$, then either $\beta_1 \omega' \in \Omega(c(\mathcal M))$ or $\beta_1 \omega' \notin \Omega(c(\mathcal M)),$ which means that $\beta_1$ completes a cycle word for $c(\mathcal{M}).$ Thus 
$$\{ \omega \in \Omega(c(\mathcal M)) : \omega \neq \ty \} = \{ \beta_1 \omega' : \beta_1\in B\times B',\omega' \in \Omega(c(\mathcal M)), \beta_1 \omega' \in \Omega(c(\mathcal M))   \}.$$ 
Therefore \eqref{eqn_3} reduces to:
\begin{equation}\label{eqn_4}
|\langle e_{c(\mathcal{M})} , e_t   \rangle|^2 = \sum\left\{ |\langle S_{\beta_1} S_{\omega'}e_{c} , e_{t}     \rangle|^2 , \omega' \in \Omega(c(\mathcal M)), \beta_1 \in B \times B' , \beta_1 \omega' \notin \Omega(c(\mathcal M))\right\}.
\end{equation}
Let $\beta_1 \omega' = \beta_1 \beta_2 \dots \beta_n.$ Note that, by \eqref{eq2.9},
\begin{equation*}
\begin{split}
\langle S_{\beta_1 \beta_2 \dots \beta_n} e_{c(\mathcal{M})} , e_t \rangle &= \langle P_{V}S_{\beta_1 \beta_2 \dots \beta_n} e_{c(\mathcal M)} , e_t \rangle \\ 
&= \langle e_{l(\beta_1) + \dots + R^{n-1} l(\beta_n) + R^n c(\mathcal M)} , e_t \rangle \prod_{k=1}^n \alpha_{\beta_{k}}.
\end{split}
\end{equation*}
But $\beta_1 \beta_2 \dots \beta_n$ is a cycle word for $c(\mathcal{M}).$ So $g_{\beta_n \dots \beta_1}(c(\mathcal{M}))=c(\mathcal{M})$ and $\beta_1 + \dots + R^{k-1} \beta_k + R^k c(\mathcal M) = c(\mathcal M).$ Thus, we have that 
$$\sum_{\substack{\omega' \in \Omega(c(\mathcal M)), \beta_1 \in B \times B' \\ \beta_1 \omega' \notin \Omega(c(\mathcal M))}} | \langle S_{\beta_1 \omega'}~e_{c(\mathcal{M})}~,~e_t \rangle|^2 =|\langle  e_{c(\mathcal{M})} , e_t  \rangle|^2 \cdot \sum_{\substack{\beta_1 \dots \beta_n \\ \textnormal{cycle word for}~c(\mathcal{M})}} \prod_{k=1}^{n} |\alpha_{\beta_k}|^2.$$
Therefore we need to show
\begin{equation}\label{eqn_5}
\sum_{\substack{\beta_1 \dots \beta_n \\ \textnormal{cycle word for}~c(\mathcal{M})}} \prod_{k=1}^{n} |\alpha_{\beta_k}|^2 = 1.
\end{equation}
Since, for all $c \in \mathcal{M},$ by \eqref{eq2.12} and Proposition \ref{pr3},
$$\sum_{c \rightarrow g_{\beta}(c)~\textnormal{is posible}}|\alpha_{\beta}|^2 = 1,$$
we can define a random walk/Markov chain $X_n, n \geq 0$ on $\mathcal{M}:$
$$P(X_n = g_{\beta}(c) | X_{n-1}=c)=|\alpha_{\beta}|^2.$$
Since the transition between any two points in $\mathcal{M}$ is possible (in several steps), the Markov chain is irreducible. Since $\mathcal{M}$ is finite, all states are recurrent (see e.g., \cite[Theorem 6.4.4]{Du10}).

So $P(\textnormal{ever reenters}~c | X_0 = c) = 1.$ But 
\begin{equation*}
\begin{split}
&P(\textnormal{ever reenters}~c | X_0 = c) =\sum_{n} P(X_n=c, X_{n-1} \neq c , \dots , X_1 \neq c | X_0 = c) \\
&=\sum_{n} \sum_{\beta_1 \dots \beta_n} P\left( X_n = g_{\beta_n}\dots g_{\beta_1}(c)=c , X_k=g_{\beta_k} \dots g_{\beta_1}(c) \neq c , 1 \leq k \leq n-1 | X_0 = c \right) \\ 
&=\sum_{\substack{\beta_1 \dots \beta_n \\ \textnormal{cycle word for}~c(\mathcal{M})}} \prod_{k=1}^{n} |\alpha_{\beta_k}|^2 .
\end{split}
\end{equation*}
Thus \eqref{eqn_5} follows and from it we get \eqref{eqn_3}.
Then we have, for $t\in\br$,
\begin{equation*}
\begin{split}
h(t) &=\sum_{\mathcal{M}} \sum_{\omega \in \Omega(c(\mathcal M))} |\langle S_{\omega} e_{c(\mathcal{M})} , e_t  \rangle|^2 \\
&=\sum_{\mathcal{M}} \sum_{\beta_1 \in B \times B'} \sum_{\omega \in \Omega(c(\mathcal M))} |\langle S_{\beta_1 \omega} e_{c(\mathcal{M})} , e_t  \rangle|^2 \\ 
&=\sum_{\mathcal{M}} \sum_{\beta_1 \in B \times B'} \sum_{\omega \in \Omega(c(\mathcal M))} |\langle S_{\omega} e_{c(\mathcal{M})} , S_{\beta_1}^{*} e_t  \rangle|^2 \\ 
&=\sum_{\beta_1 \in B \times B'} \sum_{\mathcal{M}} \sum_{\omega \in \Omega(c(\mathcal M))} |\alpha_{\beta_1}|^2 |m_{B}(g_{\beta_1}(t))|^2 \cdot |\langle S_{\omega} e_{c(\mathcal{M})} , e_{g_{\beta_1}(t)}  \rangle|^2 \\
&=\sum_{\beta_1 \in B \times B'} |\alpha_{\beta_1}|^2 |m_{B}(g_{\beta_1}(t))|^2 h(g_{\beta_1}(t)).
\end{split}
\end{equation*}
So we have the invariance equation for $h:$
\begin{equation}\label{eqn_6}
h(t)=\sum_{\beta_1 \in B \times B'} |\alpha_{\beta_1}|^2 |m_{B}(g_{\beta_1}(t))|^2 h(g_{\beta_1}(t)).
\end{equation}
Consider the interval
$$I=\left[ \frac{\min(-l(\beta))}{R-1} , \frac{\max(-l(\beta))}{R-1}   \right]$$
and note that this is invariant for the maps $g_{\beta}.$ $h$ is continuous so it has a minimum on $I.$ Let $\tilde{h}(t)=h(x)- \min_{t \in I} h(t).$ With \eqref{eqn_6} and \eqref{eq2.12} we get that $\tilde{h}$ satisfies \eqref{eqn_6} too. Also $\tilde{h}$ has a zero in $I.$ Let $Z$ be the set of zeros of $\tilde{h}$ in $I.$ If $Z$ is infinite, then, since $\tilde h$ is entire we must get $\tilde{h} \equiv 0.$ So $h$ is constant and since $h(0)=1$ we get $h \equiv 1.$ So let's assume $Z$ is finite. Using \eqref{eqn_6} for $\tilde{h}$ we have that for $t \in Z$ and $\beta \in B \times B'$ either $|\alpha_{\beta}|^2 |m_{B}(g_{\beta}(t))|^2 =0$ or $\tilde{h}(g_{\beta}(t))=0.$ Thus $Z$ is invariant. But then it contains a min-set $\mathcal{M},$ so $c(\mathcal{M}) \in Z.$ This implies that $\tilde{h}(c(\mathcal{M}))=0,$ so $c(\mathcal{M})$ is a minimum for $h.$ But $0 \leq h \leq 1$ and $h(c(\mathcal{M}))=1$ which means that $h \equiv 1$ contradicting the assumption that $Z$ is finite. Thus, $h \equiv 1$ and this means that $\| P_{\mathcal{K}} e_t \|=1=\| e_t \|$ for all $t\in\br$. So $\mathcal{K}$ contains all functions $e_t$ which by Stone-Weierstrass theorem, span the subspace $V.$ Thus $\mathcal{K} \supset V$ and then with Lemma 1.2 in \cite{DR16} we get that $P_{V} S_{\omega}e_{c(\mathcal M)},~\omega \in \Omega(c(\mathcal M)),$ where $\mathcal{M}$ is a $\min$-set, form a Parseval frame for $V.$ Using \eqref{eq2.9} and discarding the zero vectors by going back to $L$ instead of $B \times B',$ we obtain the result.

\begin{remark}\label{rem2.1}
We will show that picking a different cycle point in a min-set $\mathcal M$ will not produce a significantly different Parseval frame: the exponential functions will be the same, and the weights might be just redistributed in a sense that we will make precise below. 

Let $c$ be a cycle point in the min-set $\mathcal M$. Define
\begin{equation}
\Lambda(c):=\left\{l_{0}+R l_{1} + \dots + R^k l_k + R^{k+1} c  : l_0\dots l_k\in\Omega(c)\right\}.
\label{eq2.1.1}
\end{equation}

We will show first that $\Lambda(c)$ is the smallest set that contains $c$ and with the property that has the invariance property: $R\Lambda+L\subset\Lambda$. Indeed, if $\Lambda$ is such a set, then, since it contains $c$, one can show by induction that it contains $l_0+Rl_1+\dots+R^kl_k+R^{k+1}c$ for any $l_0,\dots,l_k$ in $L$; therefore $\Lambda$ contains $\Lambda(c)$. Then, we have to prove only that $R\Lambda(c)+L\subset \Lambda(c)$. Take $l_0\dots l_k\in\Omega(c(\mathcal M))$ and $j_0\in L$. If $j_0l_0\dots l_k$ is in $\Omega(c)$ then $j_0+R(l_0+Rl_1+\dots +R^kl_k+R^{k+1}c)$ is in $\Lambda(c)$. If $j_0l_0\dots l_k$ is not in $\Omega(c)$, this means that $j_0l_0\dots l_k$ is a cycle word for $c$. So $c=g_{l_k}\dots g_{l_0}g_{j_0}c$ which means that 
$$j_0+R(l_0+\dots +R^{k}j_k+R^{k+1}c)=c\in \Lambda(c).$$
Thus $R\Lambda(c)+L\subset \Lambda(c)$.

Now pick another point $c'$ in the min-set $\mathcal M$. We claim that $\Lambda(c)=\Lambda(c')$. To prove this, we show that $c'$ is in $\Lambda(c)$. Indeed, by Proposition \ref{pr3}, the transition from $c'$ to $c$ is possible in several steps. So there exist $l_0,\dots,l_k$ in $L$ such that $g_{l_k}\dots g_{l_0}c'=c$. This means that 
$$c'=l_0+Rl_1+\dots +R^k l_k+R^{k+1}c.$$
Since the set $\Lambda(c)$ has the invariance property, it follows that $c'\in\Lambda(c)$. But $\Lambda(c')$ is the smallest set that contains $c'$ and has the invariance property, therefore $\Lambda(c')\subset \Lambda(c)$. By symmetry, the reverse inclusion holds too and therefore, the two sets are equal. 

The equality of the two sets $\Lambda(c)$ and $\Lambda(c')$ means that the set of exponential functions that appear in the conclusion of Theorem \ref{th1.6}, is not changed if we pick a different cycle point. 

Next, we investigate what happens to the weights if we change the cycle points. Take two distinct cycle points $c\neq c'$ in $\mathcal M$. Let $l_0\dots l_k\in\Omega(c)$ so that 
$$\lambda=l_0+Rl_1+\dots+R^kl_k+R^{k+1}c\in\Lambda(c).$$
Since the transition from $c'$ to $c$ is possible, we can write $g_{j_r}\dots g_{j_0}c=c'$ and we can pick $j_0,\dots, j_r$ so that $j_0\dots j_r$ does not end in a cycle word for $c'$, which means $j_0\dots j_r$ is in $\Omega(c')$. Also $j_0\dots j_r$ is not the empty word since $c\neq c'$. Then $l_0\dots l_kj_0\dots j_r\in\Omega(c')$ and 
$$\lambda=l_0+Rl_1+\dots+R^kl_k+R^{k+1}j_0+R^{k+2}j_1+\dots+R^{k+1+r}j_r+R^{k+1+r+1}c'.$$

The weight of $\lambda$ as an element of $\Lambda(c)$ is $\alpha_{l_0}\dots\alpha_{l_k}$. The weight of $\lambda$ as an element of $\Lambda(c')$ is $\alpha_{l_0}\dots\alpha_{l_k}\alpha_{j_0}\dots\alpha_{j_r}$. 

We will show that 
\begin{equation}
\sum_{\substack{j_0\dots j_r\in\Omega(c')\\g_{j_r}\dots g_{j_0}(c)=c'}}|\alpha_{j_0}\dots\alpha_{j_r}|^2=1.
\label{eq2.1.2}
\end{equation}

For this we used the Markov chain $(X_n)_{n\in\bn}$ on $\mathcal M$ as in the Proof of Theorem \ref{th1.6}. Since this a recurrent Markov chain we have $P(X\mbox{ enters }c' | X_0=c)=1$. Then

\begin{equation*}
\begin{split}
&1=P(X \textnormal{enters}~c' | X_0 = c) =\sum_{r} P(X_{r+1}=c, X_{r} \neq c' , \dots , X_1 \neq c' | X_0 = c) \\
&=\sum_{r} \sum_{j_0 \dots j_r} P\left( X_{r+1} = g_{j_r}\dots g_{j_0}(c)=c' , X_k=g_{j_k} \dots g_{j_0}(c) \neq c' , 0 \leq k \leq r-1 | X_0 = c \right) \\ 
&=\sum_{\substack{j_0\dots j_r\in\Omega(c')\\g_{j_r}\dots g_{j_0}(c)=c'}}|\alpha_{j_0}\dots\alpha_{j_r}|^2 ,
\end{split}
\end{equation*}
and this proves \eqref{eq2.1.2}. 
\end{remark}

Thus the element $\lambda$ of $\Lambda(c)$ appears in $\Lambda(c')$ multiple times but the combined absolute value squared of the weights is the same, so the Parseval frames are essentially the same, if we change the choices of the cycle points in each min-set.

\section{Examples}

\begin{example}\label{ex1}
Let $R=4$, $B=\{0,2\}$. This is the Jorgensen-Pedersen example of a Cantor set from \cite{JP98}. We pick $L=\{0,3,15\}$ and $\alpha_{3},\alpha_{15} \in\bc\setminus\{0\}$ with $|\alpha_3|^2+|\alpha_{15}|^2=1$, $\alpha_0=1$. The assumptions 1.1 are satisfied. 

We compute now the min-sets. By Proposition \ref{pr3}, the points in a min-set are contained in $\frac12\bz$ and in the interval $[-5,0]$. We also have $m_B(x)=\frac{1+e^{2\pi i 2x}}{2}$. If $x_0$ is a point in a min-set, and $x_0=(2k+1)/2$ for some $k\in\bz$, then $m_B(x_0/4)=(1\pm i)/2\neq 0$, so the transition is possible. So $x_0/4=(2k+1)/8$ is in the min-set, but it is not in $\frac12\bz$, a contradiction. Thus the points in a min-set must be integers. Then, a simple computation shows that the min-sets are $\{0\}$ and $\{-4,-1\}$. Inside the min-set $\{-4,-1\}$ we have the possible transitions $-1\stackrel{3}{\rightarrow}-1$, $-1\stackrel{15}{\rightarrow}-4$, $-4\stackrel{0}{\rightarrow}-1$.

We have the extreme cycles: $\{0\}$ and $\{-4,-1\}$. For the extreme cycle point $0$, the cycle word is $0$. For the extreme cycle point $-1$ the cycle words are $3$ and $15.0$, and for the extreme cycle point $-4$ we have infinitely many cycle words: $0.15$, $0.3.15$, $0.3.3.15$, $0,3.3.3.15$ and so on. 

Then we can pick $0$ and $-1$ for the extreme cycle points in their respective min-sets and obtain a Parseval frame of exponential functions as in Theorem \ref{th1.6}. 

A few observations: there are elements in this Parseval frame which are not orthogonal. For example, with the notation in \eqref{eq2.1.1} take $3\in\Lambda(0)$ and $75=15+4\cdot 15\in\Lambda(0)$. Then $e_3$ and $e_{75}$ are not orthogonal:
$$\ip{e_{75}}{e_3}_{L^2(\mu_B)}=\widehat\mu_B(75-3)=\prod_{n=1}^\infty m_B\left(\frac{72}{4^n}\right)=m_B(8)m_B(2)m_B(1/2)\prod_{n=1}^\infty m_B\left(\frac{1}{2\cdot 4^n}\right)\neq 0.$$

Also $11=15+4\cdot (-1)\in\Lambda(-1)$ and $e_3$ and $e_{11}$ are not orthogonal:
$$\ip{e_{11}}{e_3}_{L^2(\mu_B)}=m_B(8/4)m_B(2/4)\prod_{n=1}^\infty m_B\left(\frac{1}{2\cdot 4^n}\right)\neq 0.$$

Note also that some of the functions might appear multiple times in the Parseval frame: for example $15$ can be written both as $15\in\Lambda(0)$ with weight $\alpha_{15}$ and as $3+4\cdot 3\in\Lambda(0)$.

\end{example}

\begin{proposition}\label{pr4.2}
Let $R=4$, $B=\{0,2\}$. Suppose $L$ and $(\alpha_l)_{l\in L}$ satisfy the assumptions 1.1. If there exist $a,b\in L$ with $a\equiv 1(\mod 4)$ and $b\equiv 3(\mod 4)$ then there are no non-trivial min-sets. 
\end{proposition}

\begin{proof}
As in Example \ref{ex1}, we can see that the points in a min-set must be integers. Suppose there is a non-trivial min-set $\mathcal M$ and let $x_0$ be a point in it. Let $a=4a_1+1$, $b=4a_2+3$.

If $x_0=4k+1$, then $g_{b}x_0=\frac{4(k-a_2)-2}{4}=\frac{2(k-a_2)-1}{2}$. Then $m_B(g_{b}x_0)\neq 0$, so the transition $x_0\rightarrow g_{b}x_0$ is possible and $g_{b}x_0\in\mathcal M$. But $g_{b}x_0$ is not an integer, a contradiction. 

If $x_0=4k+3$, then we obtain a contradiction in a similar manner, using $g_{a}x_0$. 

If $x_0=4k+2$, then we obtain a contradiction using $g_0x_0$. 

Thus, $x_0$ is of the form $x_0=4^ny_0$ with $y_0$ not divisible by $4$, $n\in\bn$. But then, the transitions $x_0\rightarrow x_0/4\rightarrow\dots\rightarrow y_0$ are possible, with digit 0 so $y_0\in \mathcal M$. But this contradicts the previous cases. 

In conclusion, there can be no non-trivial min-sets. 
\end{proof}

\begin{proposition}\label{pr4.3}
Let $R=4$, $B=\{0,2\}$. Suppose $L=\{0,a,b\}$ and $(\alpha_l)_{l\in L}$ satisfy the assumptions 1.1. Suppose $a=a_0+4a_1+4^2k$, $b=a_0+4b_1+4^2l$ for $a_0,a_1,b_1\in\{0,1,2,3\}$ and $k,l\in\bz$. If there is a non-trivial min-set, then $\{a_1,(a_1+a_0)\mod 4\}\cap \{b_1,(b_1+a_0)\mod 4\}\neq \ty$. 
\end{proposition}

\begin{proof}
We begin with a lemma. 

\begin{lemma}\label{lem4.4}
If $x_0$ is a point in a non-trivial min-set $\mathcal M$, then $x_0\in\bz$ and $x_0\equiv a_0(\mod4)$ or $x_0\equiv 0(\mod 4)$. If $x_0\equiv 0(\mod 4)$, then the only possible transition is $x_0\stackrel{0}{\rightarrow}x_0/4$. If $x_0\equiv a_0(\mod 4)$, then the only possible transitions are $x_0\stackrel{a}{\rightarrow}(x_0-a)/4$, and $x_0\stackrel{b}{\rightarrow}(x_0-b)/4$. 
\end{lemma}

\begin{proof}
We saw before that $x_0$ has to be an integer. If $x_0=4n+2$, with $n\in\bz$, then $m_B(x_0/4)=m_B((2n+1)/2)=1$ so the transition $x_0\rightarrow x_0/4$ is possible, so $x_0$ is also in $\mathcal M$, but it is not an integer, a contradiction. If $x_0=4k+d$ with $k\in\bz$, $d\in\{1,3\}\setminus \{a_0\}$, then $m_B((x_0-a)/4)=m_B((\pm 2+4(n-k))/4)=m_B((\pm 1+2(n-k))/2)=1$. So the transition $x_0\rightarrow (x_0-a)/4$ is possible. Hence $(x_0-a)/4$ is in $\mathcal M$, but it is not an integer, a contradiction.

If $x_0\equiv 0(\mod 4)$, then $x_0/4$ is an integer so $m_B(x_0/4)=1$ so the transition $x_0\stackrel{0}{\rightarrow}x_0/4$ is possible. Also $m_B((x_0-a)/4)=m_B(x_0-b)/4)=0$. So there are no other possible transitions. 

If $x_0\equiv a_0(\mod 4)$, then $(x_0-a)/4$ and $(x_0-b)/4$ are both integers and therefore the transitions $x_0\stackrel{a}{\rightarrow}(x_0-a)/4$, and $x_0\stackrel{b}{\rightarrow}(x_0-b)/4$ are both possible. Also $m_B(x_0/4)=0$. So the transition $x_0\stackrel{0}{\rightarrow}x_0/4$ is not possible. 
\end{proof}

Returning to the proof of the theorem, let $x_0$ be a point in a non-trivial min-set $\mathcal M$. We write $x_0=d_0+4d_1+16d$ for some $d_0,d_1\in\{0,1,2,3\}$ and $d\in\bz$. We claim that there is an odd number $x_0$ in $\mathcal M$. Indeed, if all the points in $\mathcal M$ are even, then with Lemma \ref{lem4.4}, the only possible transition is with digit $0$. But that would mean that $\mathcal M=\{0\}$. 
So, we can pick $x_0$ to be odd. By Lemma \ref{lem4.4}, we have $d_0=a_0$ and both transitions $x_0\stackrel{a}{\rightarrow}(x_0-a)/4$ and $x_0\stackrel{b}{\rightarrow}(x_0-b)/4$ are possible.
Then $x_1=(x_0-a)/4=d_1-a_1+4(d-k)$ and $y_1=(x_0-b)/4=d_1-b_1+4(d-l)$ are in $\mathcal M$. Using again Lemma \ref{lem4.4}, we have that $(x_1-a_1)\mod 4\in\{0,a_0\}$ and $(x_1-b_1)\mod 4\in\{0,a_0\}$. Therefore 
$x_1\in \{a_1,(a_1+a_0)\mod 4\}\cap\{b_1,(b_1+b_0)\mod 4\}$. 
\end{proof}

\begin{definition}\label{def4.5}
Let $R,B,L$ satisfy the assumptions 1.1. Let $L'$ be a non-empty subset of $L$. We say that $x_0$ is a \emph{pre-extreme cycle point} for the digit set $L'$ if there are possible transitions $x_0\stackrel{l_0'}{\rightarrow}x_1\stackrel{l_1'}{\rightarrow}\dots\stackrel{l_{n-1}'}{\rightarrow} x_n\stackrel{l_n'}{\rightarrow}\dots\stackrel{l_{n+p-1}'}{\rightarrow} x_{n+p}$, with $x_n=x_{n+p}$, $p>0$, $l_0'\dots, l_{n+p-1}'\in L'$ and $|m_B(x_i)|=1$, $i\in\{0,\dots,n+p\}$. 
\end{definition}

\begin{theorem}\label{th4.6}
Let $R=4$, $B=\{0,2\}$ and let $L=\{0,a,b\}$ and $(\alpha_l)_{l\in L}$ satisfy the assumptions 1.1. Assume $a\equiv b(\mod 4)$. Let $\mathcal M$ be a non-trivial min-set. 
Then every point $x_0$ in $\mathcal M$ is a pre-extreme cycle point for both the digit set $\{0,a\}$ and for the digit set $\{0,b\}$. In particular, there exist an extreme cycle point $c$ for the digit set $\{0,a\}$ and digits $l_0,\dots,l_{n-1}\in \{0,a\}$ such that 

$$x_0=4^nc+4^{n-1}l_{n-1}+\dots +4l_1+l_0.$$

If the digit set $\{0,a\}$ has only one extreme cycle, then there is at most one non-trivial min-set, and, given an odd extreme cycle point $c$ for the digit set $\{0,a\}$, there exist digits $j_0,\dots,j_{r-1}$ in $\{0,a\}$ such that 
$$b=-(4^{r+1}-1)c-(4^rj_{r-1}+\dots+4^2j_1+4j_0).$$
\end{theorem}

\begin{proof}
If $x_0$ is a point in a non-trivial min-set $\mathcal M$, then, by Lemma \ref{lem4.4}, when $x_0$ is even, the only possible transition is with digit $0$ and, when $x_0$ is odd, both transitions, with digits $a$ and $b$ are possible. Thus we can find possible transitions $x_0\stackrel{l_0}{\rightarrow} x_1\stackrel{l_1}{\rightarrow}\dots\stackrel{l_{n-1}}{\rightarrow} x_n\rightarrow\dots$ using only the digits from $\{0,a\}$. All the points $x_j$ are in $\mathcal M$ and, since $\mathcal M$ is finite, there is $n$ and $p>0$ such that $x_{n+p}=x_n$. By Proposition \ref{pr3}, $|m_B(x_j)|=1$ for all $j$. Thus $x_0$ is a pre-extreme cycle point for the digit set $\{0,a\}$. 
If $x_n=c$, then 
$$x_0=4x_1+l_0=4(4x_2+l_1)+l_0=\dots=4^nc+4^{n-1}l_{n-1}+\dots+4l_1+l_0.$$

If the digit set $\{0,a\}$ has only one extreme cycle, this cycle has to be in $\mathcal M$, so $c\in\mathcal M$. The transition $c\rightarrow (c-b)/4=x_0$ is possible and applying the previous equality to $x_0$ we get 
$$\frac{c-b}4=4^rc+4^{r-1}j_{r-1}+\dots+4j_1+j_0,$$
for some digits $j_0,\dots,j_{r-1}$ in $\{0,a\}$, which implies the last equality in the statement of the theorem. Since every non-trivial min-set has to contain the extreme cycle for the digits $\{0,a\}$, and since min-sets are disjoint, it follows that there can be only one such set. 
\end{proof}

\begin{corollary}\label{cor4.7}
Let $R=4$, $B=\{0,2\}$ and suppose $L=\{0,3,b\}$, $(\alpha_l)_{l\in L}$ satisfy the assumptions. Assume that there is a non-trivial min-set. Then $b$ is divisible by 3, $b$ is of the form 
\begin{equation}
b=(4^{r+1}-1)-(4^rj_{r-1}+\dots+4^2j_1+4j_0),
\label{eq4.7.1}
\end{equation}

for some digits $j_0,\dots,j_{r-1}\in\{0,3\}$, 
and every point in a non-trivial min-set is of the form 
\begin{equation}
x_0=-4^n+4^{n-1}l_{n-1}+\dots +4l_1+l_0,
\label{eq4.7.2}
\end{equation}

for some digits $l_0,\dots,l_{n-1}$ in $\{0,3\}$.
\end{corollary}

\begin{proof}
The digit set $\{0,3\}$ has only one extreme cycle $\{-1\}$ and Theorem \ref{th4.6} implies that $b$ and $x_0$ have the given form. Since $4^{n+1}-1$ is divisible by 3, it follows that $b$ is divisible by 3.
\end{proof}

\begin{example}\label{ex5.8}
We consider now $R=4$, $B=\{0,2\}$ and $L=\{0,3,b\}$ and we use Corollary \ref{cor4.7} to obtain the form of $b$, from \eqref{eq4.7.1}. We indicate in the caption the digits $j_0j_1\dots$. The points in the min-set are of the form in \eqref{eq4.7.2}. They are obtained as follows: start with $-1$, multiply by 4 and add 0 or 3 and repeat. Also the numbers,if $b>0$, should be between $-b/3$ and $0$, by Proposition \ref{pr3}. Thus, we have $-1, -4,-16,-13, -64,-61,-52,-49,\dots$.

 We illustrate the min-sets using some directed graphs: the nodes are the points in the min-sets, the edges are labeled by the digits in $L$ that make the transition possible. 

\begin{center}
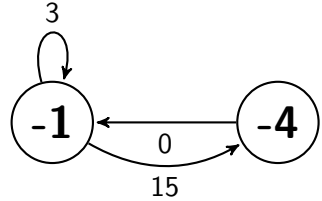
\begin{figure}[H]
\begin{tikzpicture}[->,>=stealth',shorten >=1pt,auto,node distance=3cm,
                    thick,main node/.style={circle,draw,font=\sffamily\Large\bfseries}]

  \node[main node] (1) {-1};
  \node[main node] (2) [right of=1] {-4};
  
  \path[every node/.style={font=\sffamily\small}]
    (1) edge [bend right] node[below] {15} (2)       
        edge [loop above] node {3} (1)
    (2) edge node  {0} (1) ;       
  \end{tikzpicture}
	\caption{$b=15$, $j_0j_1=03$}
\end{figure}
\end{center}

\begin{center}
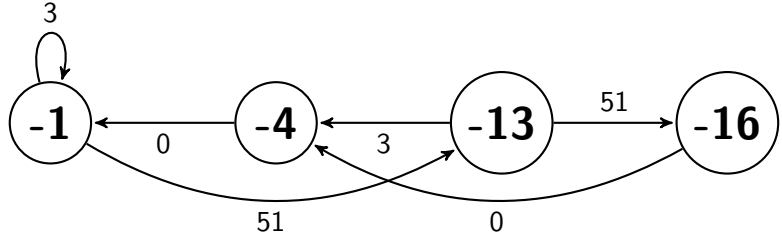
\begin{figure}[H]
\begin{tikzpicture}[->,>=stealth',shorten >=1pt,auto,node distance=3cm,
                    thick,main node/.style={circle,draw,font=\sffamily\Large\bfseries}]

  \node[main node] (1) {-1};
  \node[main node] (2) [right of=1] {-4};
  \node[main node] (3) [right of=2] {-13};
  \node[main node] (4) [right of=3] {-16};

  \path[every node/.style={font=\sffamily\small}]
    (1) edge [bend right] node[below] {51} (3)       
        edge [loop above] node {3} (1)
    (2) edge node  {0} (1)        
    (3) edge node  {3} (2)
        edge node {51} (4)
    (4) edge [bend left] node {0} (2);
\end{tikzpicture}

\caption{$b=51$, $j_0j_1=30$}
\end{figure}
\end{center}

\begin{center}
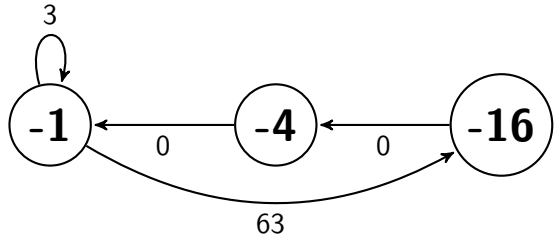
\begin{figure}[H]
\begin{tikzpicture}[->,>=stealth',shorten >=1pt,auto,node distance=3cm,
                    thick,main node/.style={circle,draw,font=\sffamily\Large\bfseries}]

  \node[main node] (1) {-1};
  \node[main node] (2) [right of=1] {-4};
  \node[main node] (3) [right of=2] {-16};

  \path[every node/.style={font=\sffamily\small}]
    (1) edge [bend right] node[below] {63} (3)       
        edge [loop above] node {3} (1)
    (2) edge node  {0} (1)        
    (3) edge node  {0} (2);
\end{tikzpicture}

\caption{$b=63$, $j_0j_1=00$}
\end{figure}
\end{center}

\begin{center}
\begin{figure}[H]
\begin{tikzpicture}[->,>=stealth',shorten >=1pt,auto,node distance=3cm,
                    thick,main node/.style={circle,draw,font=\sffamily\Large\bfseries}]

  \node[main node] (1) {-1};
  \node[main node] (2) [right of=1] {-4};
  \node[main node] (3) [right of=2] {-13};
  \node[main node] (4) [right of=3] {-16};
	\node[main node] (5) [below of=1] {-49};
	\node[main node] (6) [below of=2] {-52};
	\node[main node] (7) [below of=3] {-61};
	\node[main node] (8) [below of=4] {-64};

  \path[every node/.style={font=\sffamily\small}]
    (1) edge [bend right] node {195} (5)       
        edge [loop above] node {3} (1)
    (2) edge node  {0} (1)        
    (3) edge node  [above] {3} (2)
        edge [bend right] node [above]{195} (6)
    (4) edge [bend right] node [above] {0} (2)
		(5) edge node {3} (3)
				edge [bend right] node [below] {195} (7)
		(6) edge  node [below]{0} (3)
		(7) edge  node {3} (4)
				edge node {195} (8)
		(8) edge  node {0} (4)		
		;
		
\end{tikzpicture}

\caption{$b=195$, $j_0j_1j_2=330$}
\end{figure}
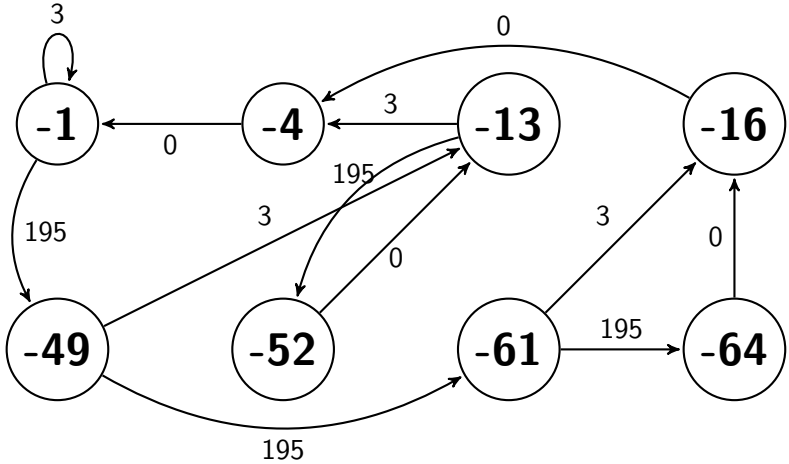
\end{center}

For $j_0j_1j_2=030$ we get $b=207$. In this case there is no non-trivial min-set. Indeed, a min-set contains the extreme cycle for $\{0,3\}$ which is $-1$. The transitions $-1\stackrel{207}{\rightarrow}-52\stackrel{0}{\rightarrow}-13\stackrel{207}{\rightarrow}-55$ are possible. So $-55$ is in the min-set, but $-55\not\equiv 3(\mod 4)$, a contradiction with Lemma \ref{lem4.4}.

For $j_0j_1j_2=300$, $b=243$, the transitions $-1\stackrel{243}{\rightarrow}-61\stackrel{243}{\rightarrow}-76\stackrel{0}{\rightarrow}-19$ are possible but $-19\not\equiv 3(\mod 4)$, a contradiction with Lemma \ref{lem4.4}
\begin{center}
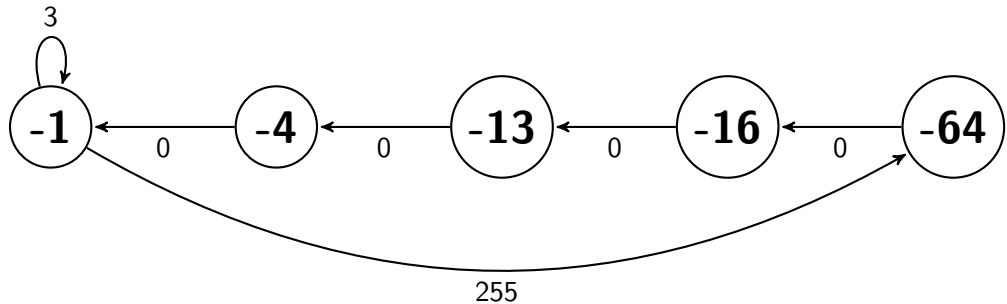
\begin{figure}[H]
\begin{tikzpicture}[->,>=stealth',shorten >=1pt,auto,node distance=3cm,
                    thick,main node/.style={circle,draw,font=\sffamily\Large\bfseries}]

  \node[main node] (1) {-1};
  \node[main node] (2) [right of=1] {-4};
  \node[main node] (3) [right of=2] {-13};
  \node[main node] (4) [right of=3] {-16};
	\node[main node] (5) [right of=4] {-64};

  \path[every node/.style={font=\sffamily\small}]
    (1) edge [bend right] node[below] {255} (5)       
        edge [loop above] node {3} (1)
    (2) edge node  {0} (1)        
    (3) edge node  {0} (2)        
    (4) edge node {0} (3)
		(5) edge node {0} (4);
\end{tikzpicture}

\caption{$b=255$, $j_0j_1j_2=000$}
\end{figure}
\end{center}

\begin{proposition}\label{pr4.9}
Let $R=4$, $B=\{0,2\}$ and $L=\{0,3,4^{n+1}-1\}$, for some $n\in\bn$. Then there is a min-set $\{-1,-4,-4^2,\dots, -4^{n}\}$. 
\end{proposition}

\begin{proof}
Indeed $(-1-(4^{n+1}-1))/4=-4^n$, $(-1-3)/4=-1$ and $(-4^k)/4=4^{k-1}$ so we see that we have a min-set for the digit set $L$. 
\end{proof}
\end{example}

 \noindent {\it Acknowledgments}.  This work was partially supported by a grant from the Simons Foundation (\#228539 to Dorin Dutkay)

\end{document}